\documentclass{article}
\usepackage[utf8]{inputenc}
\usepackage{amsfonts}
\usepackage{amsmath}
\usepackage{amssymb}
\usepackage{amsthm}
\usepackage{xcolor}
\usepackage{bm}
\usepackage{mathtools}
\usepackage{enumitem}
\usepackage{verbatim}
\usepackage{pinlabel}
\usepackage{caption}
\usepackage{subcaption}
\usepackage{tikz}
\usepackage{graphicx}
\usepackage{hyperref}
\usepackage{tikz-cd}

\newtheorem{theorem}{Theorem}[section]

\newtheorem{proposition}[theorem]{Proposition}
\newtheorem*{thm}{Main theorem}

\newtheorem{corollary}[theorem]{Corollary}

\newtheoremstyle{claim}
  {\topsep}
  {\topsep}
  {}
  {}
  {\itshape}
  {.}
  {.5em}
  {\thmname{#1}\thmnumber{ #2}\thmnote{ (#3)}}
\theoremstyle{claim}

\DeclareMathOperator{\interior}{int}
\DeclareMathOperator{\id}{id}

\DeclareMathOperator{\MCG}{MCG}

\DeclareMathOperator{\Aut}{Aut}

\newcommand{\p}[1]{\medskip\noindent\textbf{#1}\textbf{.}}

\definecolor{alexRed}{RGB}{221, 0, 0}
\definecolor{alexBlue}{RGB}{0, 0, 244}
\definecolor{alexGreen}{RGB}{0, 177, 0}
\definecolor{alexPurple}{RGB}{191, 0,255}

\definecolor{algRed}{RGB}{255, 0, 0}
\definecolor{algBlue}{RGB}{0, 75, 255}
\definecolor{algGreen}{RGB}{0, 132, 0}

\definecolor{triRed}{RGB}{249, 0, 0}
\definecolor{triBlue}{RGB}{20, 0, 255}
\definecolor{triGreen}{RGB}{0, 149, 0}

\definecolor{exhaustionpink}{RGB}{255,0,255}

\title{An Alexander method for infinite-type surfaces}
\author{Roberta Shapiro}
\date{}

\begin{document}

\maketitle

\begin{abstract}
The Alexander method is a combinatorial tool used to determine when two elements of the mapping class group are equal. We extend the Alexander method to include the case of infinite-type surfaces. Versions of the Alexander method were proven by Hern\'andez--Morales--Valdez and Hern\'andez--Hidber. As sample applications, we verify a particular relation in the mapping class group, show that the centralizers of many twist subgroups of the mapping class group are trivial, and provide a simple basis for the topology of the mapping class group.
\end{abstract}

\section{Introduction}

The Alexander method is a tool used to determine whether two homeomorphisms of a surface $S$ are isotopic. 
It states that two homeomorphisms of $S$ are isotopic if and only if they have identical actions on the isotopy classes of the curves in an Alexander system, which is a collection of curves with certain properties described below. 

The \textit{mapping class group} of a surface $S$, denoted $\MCG(S)$, is the group of isotopy classes of homeomorphisms of $S$. (When $S$ is orientable, this group is usually called the extended mapping class group; this distinction will be of no consequence in this paper.) There is an analogy between mapping class groups and general linear groups. Via the Alexander method, Alexander systems for surfaces play the role of basis vectors in a vector space in the sense that isotopy classes of homeomorphisms are determined by their actions on such a set of curves. 

Most earlier results surrounding the Alexander method work with surfaces of finite-type--that is, surfaces whose fundamental groups are finitely generated. The most general statement of the Alexander method for finite-type surfaces was formulated by Farb--Margalit \cite[Proposition 2.8]{primer}, but the method can be traced back to the works of Dehn \cite{Dehn,DehnTranslation} and Thurston \cite{FLP,FLPTranslation}. 

In this paper, we extend the notion of an Alexander system to include curve systems on infinite-type surfaces (surfaces whose fundamental groups are infinitely generated), and then prove our main result: a generalization of the Alexander method of Farb--Margalit to include the case of infinite-type surfaces (including non-orientable surfaces) \cite{primer}. 

A version of the Alexander method was proven for orientable surfaces by Hern\'andez--Morales--Valdez \cite{HMV} and for non-orientable surfaces by Hern\'andez--Hidber \cite{HH}. In both of these papers, the authors construct a family of Alexander systems such that any homeomorphism that permutes the isotopy classes of curves and arcs in an Alexander system by the identity is isotopic to the identity homeomorphism. In this paper, we prove the Alexander method in the case of arbitrary permutations and arbitrary Alexander systems. As such, our main theorem is a more direct analogue of the Alexander method for finite-type surfaces given by Farb--Margalit \cite{primer}.

\medskip
\noindent \textbf{Alexander systems.} 
Let $S$ be a (possibly infinite-type and possibly non-orientable) surface. Although $S$ may have punctures, it will be convenient to treat these punctures as marked points. For the remainder of this paper, we will refer to marked points and punctures (isolated planar ends) interchangeably. We define a curve to be the image of an embedding $\gamma:S^1\hookrightarrow S$. In this paper, an \textit{arc} is the image of a proper embedding $\gamma:(0,1)\hookrightarrow S$ with the additional characteristic of being \textit{finite-type}---that is, every arc has a regular neighborhood that is a compact surface (treating punctures as marked points). 

A curve in $S$ is \textit{essential} if it is not nullhomotopic, not homotopic to a puncture, and not homotopic to the boundary of a M\"{o}bius band; \textit{simple} if it does not self-intersect; and \textit{non-peripheral} if it is not isotopic to a boundary component. Our definition of curves implies that all curves are simple. An arc is essential and non-peripheral if it and (a subset of) a boundary component do not jointly bound an unpunctured disk in $S$. 

A \textit{subsurface exhaustion} of surface $S$ is a sequence of finite-type subsurfaces $\{S_n\}$ with $S_i\subset S_{i+1}$ such that $\cup S_n=S$. 

We say that a collection $\Gamma=\{\gamma_i\}_{i\in I}$  of essential, non-peripheral, simple closed curves and finite-type arcs on $S$ is an \textit{Alexander system} if it satisfies the following properties:
\begin{enumerate}
    \item (\textit{minimal position}) $\gamma_i$ and $\gamma_j$ are in minimal position for all $i,j\in I,$
    \item (\textit{distinct isotopy classes}) no two elements of $\Gamma$ are isotopic,
    \item (\textit{no triple intersections}) for distinct $i,j,k\in I,$ at least one of $\gamma_i\cap \gamma_j,$ $\gamma_i\cap \gamma_k,$ or $\gamma_j\cap \gamma_k$ is empty, and
	\item (\textit{local finiteness}) for a fixed subsurface exhaustion $\{S_n\}$ of $S$, \[|\{\gamma_i \in \Gamma\ :\ \gamma_i\cap S_n \neq \emptyset\}|<\infty\] for all $n$.
\end{enumerate}

Properties (1), (2), and (3) of Alexander systems are inherited from the classical finite-type Alexander method. Property (4) is only needed for infinite-type surfaces and is automatically satisfied for finite-type surfaces. Property (4) is also equivalent to the following: every finite-type subsurface of $S$ intersects finitely many elements of $\Gamma.$ In fact, property (4) is a necessary condition; a counterexample to the Alexander method with property (4) not satisfied is given in Section~\ref{sec:nonexamples}.

We say that a set of curves and arcs in minimal position in $S$ \textit{fills} $S$ if $S\setminus \Gamma$ is a union of disks, once punctured disks, and M\"obius bands (each possibly with noncompact boundary). Equivalently, the components of $S\setminus \Gamma$ have a trivial mapping class group.

\medskip
\noindent\textbf{Statement of the main theorem.} To state the main theorem, we require several more definitions. 
Define a surface graph to be an abstract graph with an embedding into some surface. Given an Alexander system on a surface $S$, let $G(S,\Gamma)=\cup \gamma_i$ be the surface graph (an abstract graph with an embedding into a surface) in $S$ whose vertex set, denoted $V(G),$ is comprised of the points of intersection of curves and arcs in $\Gamma$ and the endpoints of arcs (using marked points in lieu of punctures), and whose edges are the connected components of $\cup \gamma_i \setminus V(G)$. Let $G'(S,\Gamma)$ be the barycentric subdivision of $G(S,\Gamma)$, which is also a surface graph. 

An automorphism of a surface graph is an automorphism of the associated abstract graph that arises from a homeomorphism of the surface.

\begin{thm}
  Let $S$ be any surface, $\Gamma=\{\gamma_i\}_{i\in I}$ an Alexander system in $S$, and $\phi:S\to S$ a homeomorphism. Suppose $\sigma\in \Sigma_I$ is a permutation such that $\phi(\gamma_i)$ is isotopic to $\gamma_{\sigma(i)}$ for all $i.$ Then,

\begin{enumerate}
	\item  there exists a homeomorphism $\psi:S\to S$ isotopic to $\phi$ rel $\partial S$ such that $\psi(\gamma_i)=\gamma_{\sigma(i)}$ for all $\gamma_i\in \Gamma,$ 
	\item $\phi$ induces a unique automorphism $\phi_*$ of $G'(S,\Gamma),$ and 
	\item if $\Gamma$ is filling, then $\phi_*$ is the identity if and only if $\phi$ is isotopic to the identity.
\end{enumerate}
\end{thm}

The statement that $\phi_*$ acts by the identity on the barycentric subdivision of $G(S,\Gamma)$ is equivalent to saying that $\phi$ induces the identity automorphism of $G(S,\Gamma)$ and preserves the orientation of each loop edge.

We further note that $\phi$ need not be a self-homeomorphism of $S$, as our proof will not use this at all. That is, our proof applies to a homeomorphism $\phi:S\to S'$ and Alexander systems $\Gamma$ (indexed by $I$) and $\Gamma'$ (indexed by $I'$) on $S$ and $S'$, respectively, with a bijection $\sigma:S\to S'.$ In the case of distinct surfaces, $\phi_*$ is a graph isomorphism and the last statement of the main theorem does not apply.

When we work with the domain of $\phi,$ we will say we are working with the source and when we work with the range of $\phi$, we will name it the target. 

\p{A useful version: the case of the identity mapping class} A \textit{stable} Alexander system is an Alexander system such that any mapping class that acts by the identity on the set of isotopy classes of curves in this system is the identity mapping class. 

The existence of a stable Alexander system is proved for orientable infinite-type surfaces by Hern\'andez--Morales--Valdez \cite{HMV} and for non-orientable surfaces by Hern\'andez--Hidber \cite{HH}, both without noncompact boundary components. Both papers provide explicit constructions for stable Alexander systems. Given the constructions, proof of stability directly follows from the Alexander method for finite-type surfaces. The existence of a stable Alexander system is further proved for surfaces with noncompact boundary by Dickmann \cite{Dickmann}. 

In this paper, we provide another criterion for determining whether a homeomorphism is isotopic to the identity based on how it acts on an Alexander system, but this time without a significant restriction on the Alexander system; this is item (3) in the Alexander method. 

\medskip\noindent\textbf{Applications of the Alexander method.} The Alexander method and the existence of stable Alexander systems are key to proving multiple fundamental results about mapping class groups of finite-type surfaces, including the computation of the center of the mapping class group, the Dehn--Nielsen--Baer theorem, the solvability of the word problem for the mapping class group, and the existence of certain relations in the mapping class group \cite{primer}. 

Some of the above applications were extended to mapping class groups of infinite-type surfaces. For instance, Lanier and Loving \cite{lanierloving} compute the center of the mapping class group of a surface of infinite type using the results of Hern\'andez--Morales--Valdez; their approach is analogous to the finite-type case. 

In Section \ref{sec:apps} of this paper, we include several sample applications of our Alexander method: verifying relations in the mapping class group, computing centralizers of subgroups of the mapping class group, and describing the topology of the mapping class group using a simpler basis.

\medskip\noindent\textbf{Paper outline.} In Section \ref{sec:apps}, we prove the applications of our main theorem from the previous paragraph. In Section \ref{sec:proof}, where we prove the main theorem. We conclude with Section \ref{sec:nonexamples}, where we provide counterexamples to the Alexander method with each of the hypotheses altered.

\bigskip
\noindent\textit{\textbf{Acknowledgments.}}
The author is infinitely grateful to Dan Margalit for his support and for many helpful conversations. The author would further like to thank Jes\'us Hern\'andez Hern\'andez for extensive disucssions on a draft of this paper. The author would also like to thank Ferr\'an Valdez for correspondences regarding the subject matter and Daniel Minahan, Jorge Aurelio V\'iquez, Katherine Booth, Noah Caplinger, Yvon Verberne, Wade Bloomquist, and Brian Pinsky for many discussions. Thank you to Justin Lanier, John Etnyre, and Jennifer Hom for comments on a draft of this paper. Finally, thank you to an anonymous referee for pointing out a hole in an earlier draft of this paper. The author is supported by NSF grant DMS 1745583.

\section{Sample applications}\label{sec:apps}

In this section, we discuss three applications of the Alexander method: proving relations in the mapping class group of a surface, finding the centralizer of subgroups of the mapping class group, and describing the topology (and, more precisely, the non-discreteness and local non-compactness) of the mapping class group using a simpler basis. These applications represent the many types of questions the Alexander method can answer.

\subsection{Relations in the mapping class group}

Let $S$ be the surface with three ends accumulated by genus and no boundary components nor punctures (the tripod surface). We label the ends $a,\ b,$ and $c,$ as in Figure \ref{fig:exampleapp1shifts} and let $\Gamma$ be the Alexander system pictured in Figure \ref{fig:exampleapp1shifts}. 

\begin{figure}[htp]

\begin{center}
		\begin{tikzpicture}

		\node[anchor=south west,inner sep=0] at (2.4,0) {\includegraphics[width=0.6\textwidth]{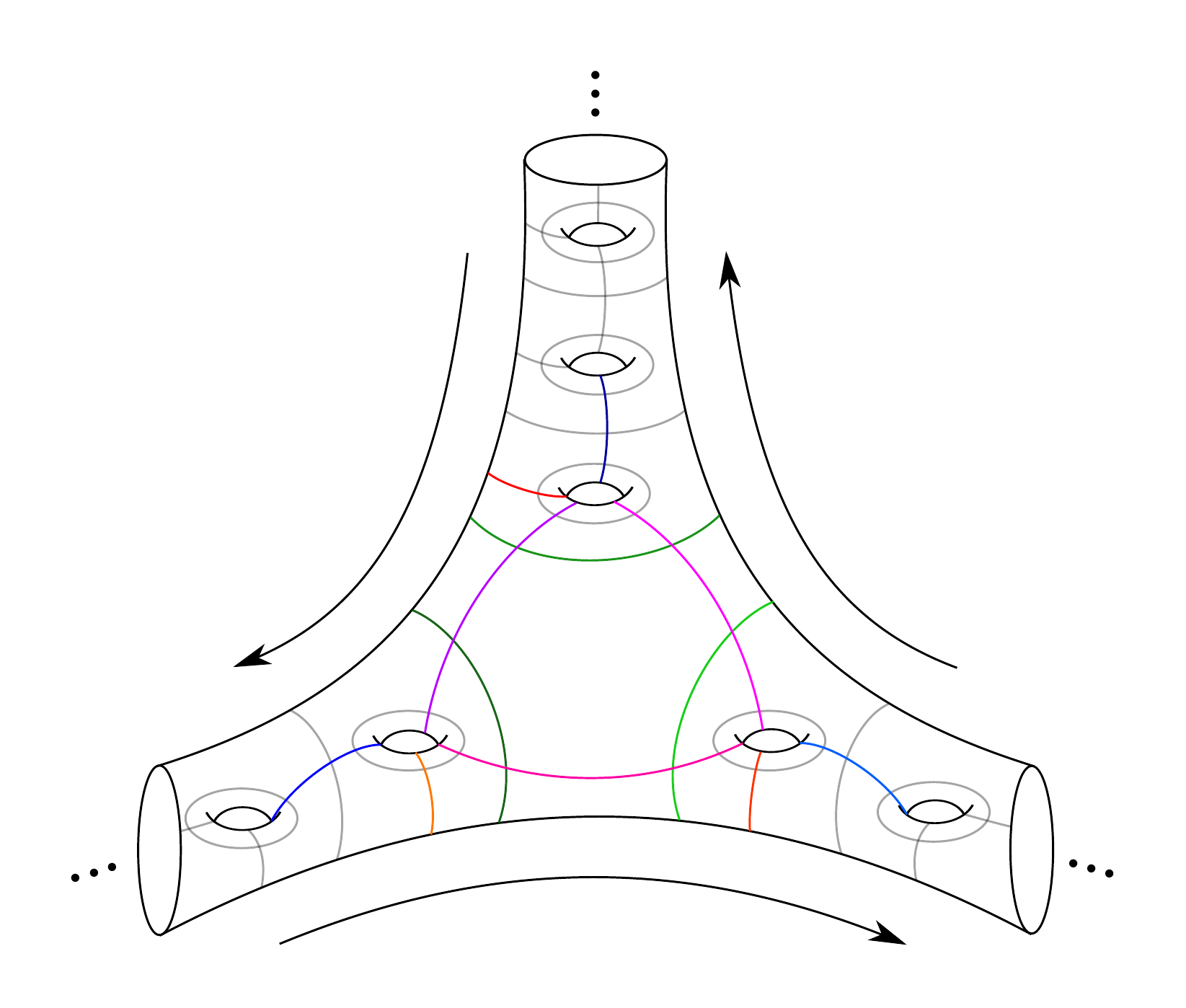}};


\node at (6.3,5.6){$b$};
\node at (3,1.1){$c$};
\node at (9.1,1.1){$a$};

\node at (6,0.5){$h_{ca}=h_{ac}^{-1}$};
\node at (4.7,3.5){$h_{bc}$};
\node at  (7.45,3.5){$h_{ab}$};

		\end{tikzpicture}

		\end{center}

    \caption{Surface $S$ with ends $a,\ b,$ and $c$, along with an Alexander system $\Gamma$ on $S$. The handle shifts used in this example are pictured as well.}
    \label{fig:exampleapp1shifts}
\end{figure}
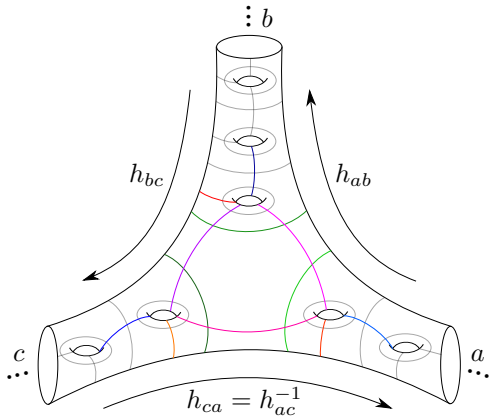

\begin{figure}[htp]

\begin{center}
		\begin{tikzpicture}

		\node[anchor=south west,inner sep=0] at (0,0) {\includegraphics[width=\textwidth]{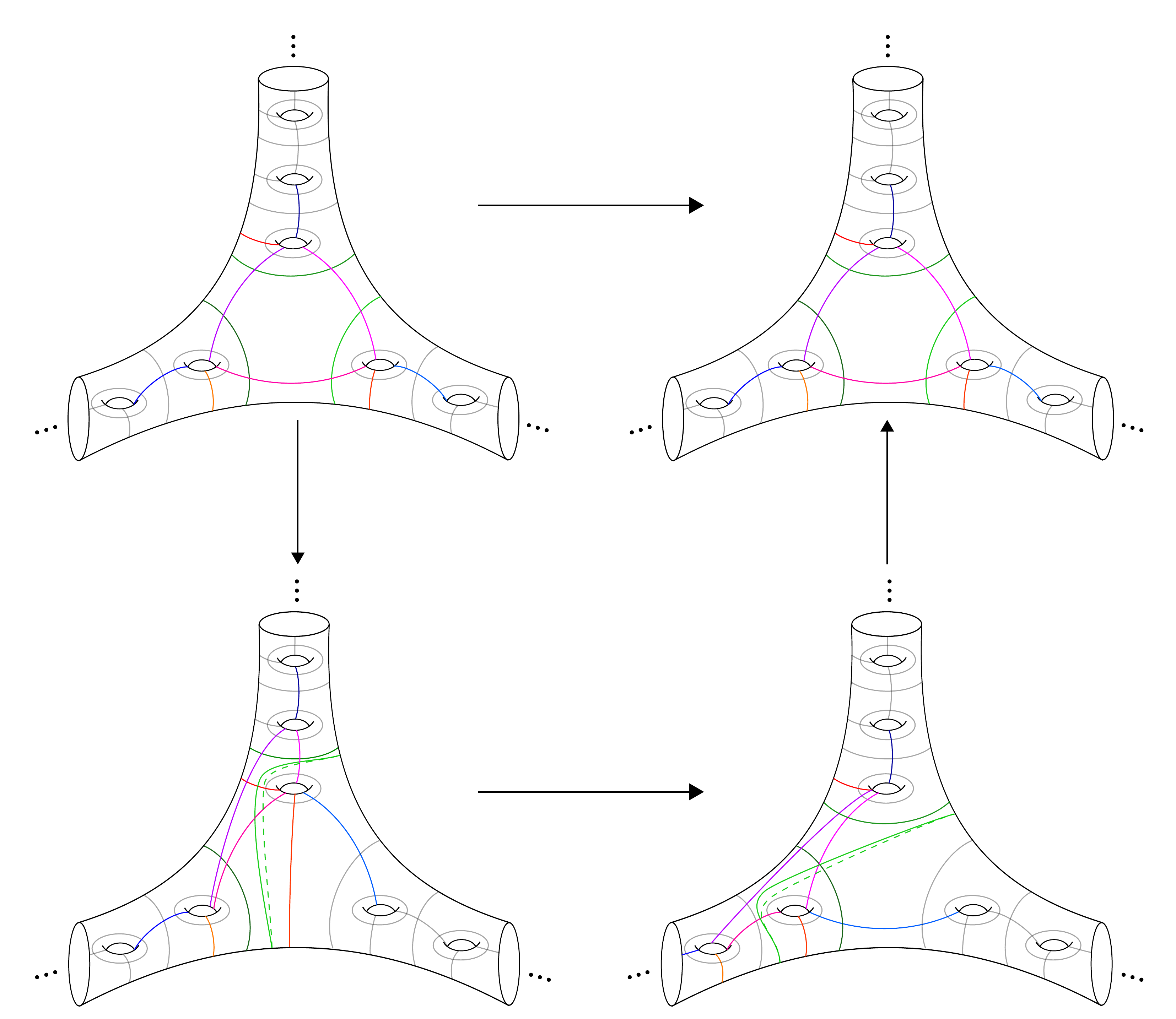}};


\node at (5.55,6.5){$a$};
\node at (.4,6.5){$c$};
\node at (3.2, 10.1){$b$};

\node at (6,8.65){$id$};
\node at (2.8,5.5){$h_{ab}$};
\node at (6, 2.6){$h_{bc}$};
\node at (9.57,5.5){$h_{ca}$};

		\end{tikzpicture}

		\end{center}

    \caption{Computations for the action of $h_{ca}h_{bc}h_{ab}$ on $S$ and $\Gamma$.}
    \label{fig:exampleapp1}
\end{figure}
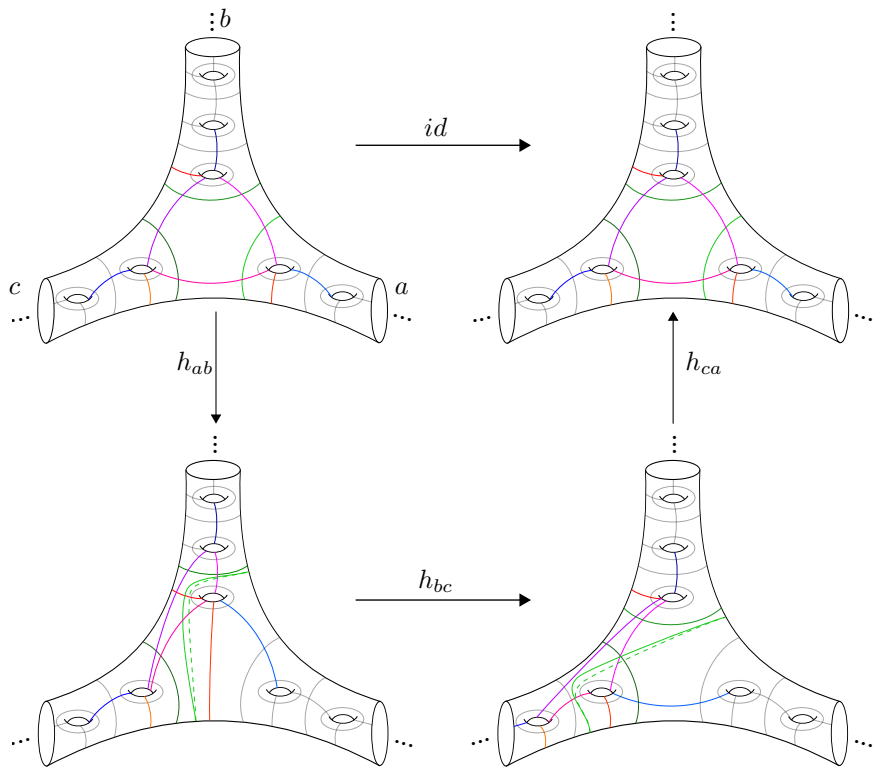

Let $h_{ab}$ be a handle shift from end $a$ to end $b$, $h_{bc}$ a handle shift from end $b$ to end $c,$ and $h_{ca}$ a handle shift from end $c$ to end $a.$ (Handle shifts are not determined by a pair of ends; we take the arrows in Figures \ref{fig:exampleapp1shifts} and \ref{fig:exampleapp1} as the definitions of the handle shifts.) We verify a the following relation, which is a special case of a relation initially proven by Afton--Freedman--Lanier--Yin \cite{AFLY} using the Alexander method for finite-type surfaces.

\begin{proposition}[Afton--Freedman--Lanier--Yin]
Let $S$ be the surface and $h_{ab},\ h_{bc}$ and $h_{ca}$ the handle shifts in Figure \ref{fig:exampleapp1}. Then, 
\[h_{ca}h_{bc}h_{ab}=id,
\]
where we compose from right to left.
\end{proposition}

The computations are shown in Figure \ref{fig:exampleapp1}. We only need to keep track of the curves within one genus of the center of $S$, as all other curves are not distorted by the handle shifts. 

The orientation of loop edges in $G(S,\Gamma)$ does not change upon application of $h_{ca}h_{bc}h_{ab}$. We conclude that $h_{ca}h_{bc}h_{ab}=id$ is a valid relation. 

Afton--Freedman--Lanier--Yin also observe that in the tripod surface, $h_{ab},\ h_{bc},$ and $h_{ca}^{-1}$ are conjugate in $\MCG(S)$, implying that in the abelianization of the $\MCG(S),$ handle shifts are equal to the identity \cite{AFLY}.

Similar relations are true in more general settings, including non-orientable handle shifts. Such relations can be verified using the version of the Alexander method in Hern\'andez--Morales--Valdez \cite{HMV}; our Alexander method provides more flexibility with the choice of the Alexander system. We note that the action on the above Alexander system alone (without considering the orientations of curves) does not show that the above homeomorphism is isotopic to the identity, as there is a second, orientation-reversing homeomorphism of $S$ that fixes the isotopy classes of all curves in $\Gamma.$ 

\subsection{Centralizer of subgroups of mapping class groups}

In this section, we will show a result relating to the centralizers of twist subgroups of mapping class groups.

Let $S$ be a surface with filling Alexander system $\Gamma=\{\gamma_i\}_{i\in I}$ comprised of two-sided curves. We define $T=\langle \{T_{\gamma_i}^{k_i}\}_{i\in I} \rangle$ to be the subgroup generated by the $k_i$th powers of the Dehn twists about the $\gamma_i$. The propositions in this section concern this subgroup.

Let $G_r(S,\Gamma)$ be the ribbon graph associated with $G(S,\Gamma)$. Since $\Gamma$ is comprised of two-sided curves without triple intersections, we have that the geometric realization of $G_r(S,\Gamma)$ is a union of annuli $\{A_i\}$ with core curves $\{\gamma_i\}$ glued together along disks. 
We define an automorphism of a ribbon graph to be an automorphism of the underlying abstract graph that preserves the $A_i$ setwise. We note that this is not the usual definition of an automorphism of a ribbon graph.

If $f:S\to S$ induces an automorphism of $G_r(S,\Gamma)$ that fixes the curves corresponding the elements of $\Gamma$, then each $A_i$ is mapped to itself and its orientation is either preserved or reversed. In particular, for the orientation of $A_i$ to be preserved, the core curve $\gamma_i$ of $A_i$ must be mapped to itself and the orientations of both the core and cocore of $A_i$ must both be preserved or must both be reversed.

Let $G'_a(S,\Gamma)$ be the abstract graph associated with $G'(S,\Gamma).$ We note that automorphisms of $G'(S,\Gamma)$ induce automorphisms of $G_r(S,\Gamma)$ by considering $G_r(S,\Gamma)$ to be embedded in $S$ with core curves $\gamma_i$. Similarly, automorphisms of $G_r(S,\Gamma)$ induce automorphisms of $G'_a(S,\Gamma)$ by forgetting the annular structure. That is, there are natural inclusion (and, in particular, injective) maps
\[\Aut(G'(S,\Gamma)) \xhookrightarrow{F} \Aut(G_r(S,\Gamma)) \xhookrightarrow{H} G'_a(S,\Gamma).\]
We note that if we replace $\Aut(G'(S,\Gamma))$ by $\Aut(G(S,\Gamma))$, this map need not be injective.

We call an Alexander system $\Gamma$ on $S$ \textit{weakly stable} if: 1) it is filling and 2) any automorphism of $G_r(S,\Gamma)$ in the image of $F$ above that preserves each $A_i$ with orientation is the image of the trivial automorphism of $G'(S,\Gamma)$. 

\begin{proposition}\label{prop:appcentralizer}
Let $S$ be a surface and $\Gamma$ a weakly stable Alexander system comprised of two-sided curves. Let $\{k_i\}_{i\in I}$ be a collection of non-zero integers. Then, the centralizer in $\MCG(S)$ of the subgroup $T=\langle \{T_{\gamma_i}^{k_i}\}_{i\in I} \rangle$ is trivial.
\end{proposition}

It is not straightforward to see how Proposition \ref{prop:appcentralizer}---and, by extension, Corollary~\ref{cor:appcentralizer} that follows---could be achieved using the version of the Alexander method of Hern\'andez--Morales--Valdez \cite{HMV}, as the graphs $G(S, \Gamma)$ and $G_r(S,\Gamma)$ play a crucial role in the proofs.

\begin{proof}[Proof of Proposition \ref{prop:appcentralizer}]
Suppose $[f]\in \MCG(S)$ is in the centralizer of $T$. Since $[f]$ commutes with each $T_{\gamma_i}^{k_i},$ it follows that $f$ fixes the isotopy classes of all curves in $\Gamma$ as well as the orientation of the annular neighborhood of each curve. We then have that $f$ induces an automorphism $\tilde{\alpha}$ of $G'(S,\Gamma)$ by statement 2 of the Alexander method. Furthermore, $\alpha=F(\tilde{\alpha})$ is an automorphism of $G_r(S,\Gamma)$ that preserves the orientations of the $A_i$. 

 Since $\Gamma$ is weakly stable, $\tilde{\alpha}$ is the identity. The third statement of the Alexander method implies that $[f]$ is the identity mapping class.
\end{proof}

We use the above proposition to prove the following example.

\begin{corollary}\label{cor:appcentralizer}
Let $S$ and $\Gamma$ be the surface and Alexander system in Figure~\ref{fig:appcentsystem}. Let $\{k_i\}_{i\in I}$ be a collection of non-zero integers. Then, $\Gamma$ is weakly stable and hence the centralizer in $\MCG(S)$ of the subgroup $T=\langle \{T_{\gamma_i}^{k_i}\}_{i\in I} \rangle$ is trivial.
\end{corollary}

\begin{figure}[htp]

\begin{center}
		\begin{tikzpicture}

		\node[anchor=south west,inner sep=0] at (0,0) {\includegraphics[width=\textwidth]{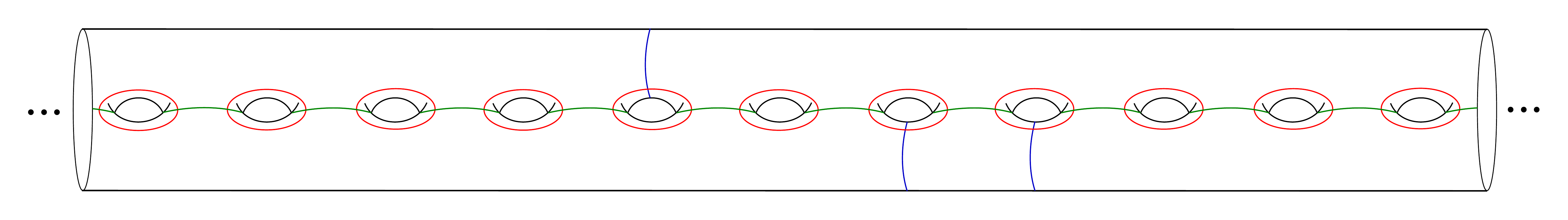}};




		\end{tikzpicture}
		\end{center}

    \caption{The Alexander system $\Gamma=\{\gamma_i\}_I$ in Corollary~\ref{cor:appcentralizer}.}
    \label{fig:appcentsystem}
\end{figure}

\begin{figure}[htp]

\begin{center}
		\begin{tikzpicture}

		\node[anchor=south west,inner sep=0] at (0,0) {\includegraphics[width=\textwidth]{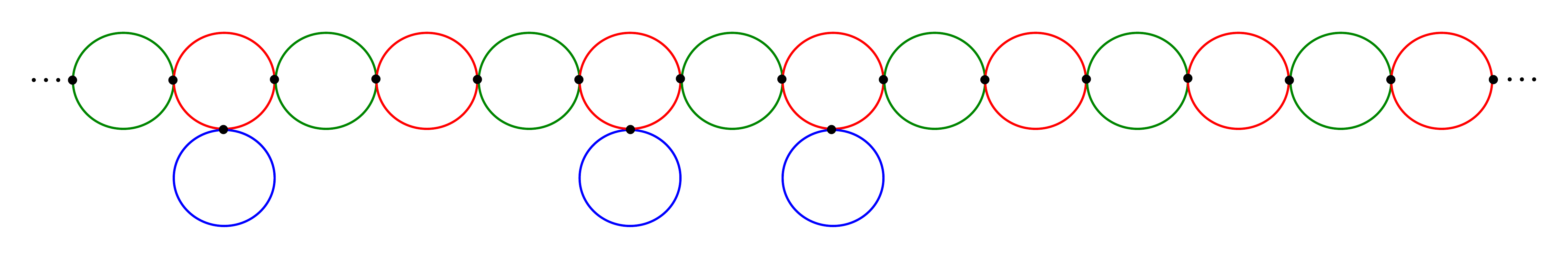}};
		\end{tikzpicture}
		\end{center}

    \caption{The graph $G_a(S,\Gamma)$. Any automorphisms of $G_a(S,\Gamma)$ must preserve loop edges. Due to the lack of symmetry in $G_a(S,\Gamma),$ every automorphism of $G(S,\Gamma)$ must fix all loop edges, possibly reversing orientation. As such, all monochromatic loops must be fixed setwise.}
    \label{fig:appcentgraph}
\end{figure}

\begin{proof}
Let $\alpha\in\Aut G_r(S,\Gamma)$ be induced by an automorphism $\tilde{\alpha}$ of $G'(S,\Gamma)$ and suppose $\alpha$ maps the annuli $A_i$ to themselves and preserves their orientations. Let $\alpha_*\in\Aut G'_a(S,\Gamma)$ be the induced automorphism of the abstract graph. That is,
\begin{alignat*}{3}
\Aut(G'(S,\Gamma)) &\xhookrightarrow{F} \Aut(G_r(S,\Gamma)) &\xhookrightarrow{H} G'_a(S,\Gamma)\\
\tilde{\alpha}\quad&\mapsto \quad\quad\quad\alpha\quad &\mapsto\quad\alpha_*.\quad
\end{alignat*}
We note that $\tilde{\alpha}$ induces $\alpha_*$ as well by forgetting the surface structure. We want to show that $\tilde{\alpha}$ is the identity automorphism. 

Let $G_a(S,\Gamma)$ be the abstract graph associated to $G(S,\Gamma)$. If we orient the loop edges of $G_a(S,\Gamma)$ and keep track of their orientations, then the elements of $\Aut(G_a(S,\Gamma))$ are in a natural bijection with elements of $\Aut(G'_a(S,\Gamma))$ (automorphisms of $G_a(S,\Gamma)$ need not repsect orientation). Let $\alpha'_*\in \Aut(G_a(S,\Gamma))$ be the automorphism corresponding to $\alpha_*.$ 

 We first show that $\alpha'_*$ is a product of automorphisms that reverse the directions of loop edges or swap pairs of edges in 2-cycles. Any automorphism of $G_a(S,\Gamma)$ must preserve same-colored cycles setwise such that no two cycles are interchanged. This property of $\alpha'_*$ is also inherited from $\alpha.$ Moreover, all 3-cycles are fixed by any automorphism of $G_a(S,\Gamma).$ It follows that $\alpha_*$ may only reverse orientations of loop edges (equivalent to $a_*$ swapping edges in 2-cycles in $G'_a(S,\Gamma)$) or swap pairs of edges in 2-cycles.

Since $\alpha$ preserves the $A_i$ with orientations, the orientations of all curves that intersect the curves corresponding to the red 3-cycles are preserved. As such, the orientations of loop edges cannot be reversed. Similarly, by working our way sequentially toward the ends, the orientations of all other curves must be preserved as well. It follows that $\alpha'_*$---and similarly $\alpha_*$---are the trivial automorphisms. It follows from injectivity that $\tilde{\alpha}$ is the trivial automorphism of $G'(S,\Gamma)$.

We now have that $\Gamma$ is weakly stable and Proposition~\ref{prop:appcentralizer} gives us that the centralizer of $T$ is trivial. 
\end{proof}

\subsection{Topology of mapping class groups}

Aramayona--Vlamis discuss the topology on the mapping class group of infinite-type surfaces and show that the mapping class groups of infinite-type surfaces are not discrete and not locally compact \cite{AV}. As we describe below, the topology can be described in a simpler manner using the Alexander method. The applications in this section were suggested by Jes\'us Hern\'andez Hern\'andez \cite{hernandez}.

\medskip\noindent\textbf{The topology on the mapping class group via the Alexander method.}
There is a natural topology on $\MCG(S)$ arising from the compact-open topology. The permutation topology is an equivalent and more combinatorial description of the topology on mapping class groups of infinite-type surfaces (see \cite[Section 4.1]{AV} for more details). We give a new perspective on the permutation topology in light of the Alexander method.

For any finite set $A$ of isotopy classes of curves in $S$, let $U(A)$ be the set of mapping classes that fix all of the curves in $A.$ Define the permutation topology on $\MCG(S)$ to be the topology whose basis elements are all $\MCG(S)-$translates of all $U(A)$.

Let $S$ be an infinite-type surface and $\Gamma$ a stable Alexander system in $S.$ Such a system exists by the works of Hern\'andez--Morales--Valdez \cite{HMV}, Hern\'andez--Hidber \cite{HH}, and Dickmann \cite{Dickmann}. We define the \textit{Alexander topology} $\mathcal{T}_\Gamma$ on $\MCG(S)$ as follows. For any finite subset $B$ of (the isotopy classes of) curves in $\Gamma,$ let $U(A)$ be the set of mapping classes that fix all the curves in $B.$ Then, the basis for $\mathcal{T}_A$ is the set of all $\MCG(S)-$translates of all $U(B).$

$\mathcal{T}_\Gamma$ is indeed a topology. The basis elements cover $\MCG(S)$ since $h\in h\cdot U(B)$ for every $B\subseteq \Gamma$ (abusing notation, as the elements of $B$ are isotopy classes of curves and $\Gamma$ contains curves). We also have that if $h\in g_1\cdot U(B_1) \cap g_2\cdot U(B_2),$ then $h\in h\cdot(B_1)\cap h\cdot(B_2)=h\cdot (B_1\cup B_2),$ another basis element. 

We show that the permutation topology and the Alexander topology are equivalent. This result is to be expected since, by the Alexander method, a mapping class is determined by its action on an Alexander system.

\begin{proposition}\label{prop:topology}
Let $S$ be an infinite-type surface and $\Gamma$ a stable Alexander system on $S$. Then, $\mathcal{T}_\Gamma$ is equal to the permutation topology $\mathcal{T}$.
\end{proposition}

\begin{proof}
We have that $\mathcal{T}$ is a priori finer than $\mathcal{T}_\Gamma$ since $\mathcal{T}_\Gamma\subset\mathcal{T},$ so it remains to show the opposite inclusion. Let $h\in \MCG(S)$ and let $A$ be a set of isotopy classes curves such that $g\cdot U(A)$ is a basis element that contains $h$. Then, $h(\alpha)=g(\alpha)$ for $\alpha\in A,$ so $g\cdot U(A)=h\cdot U(A).$ Abusing notation, let $\Gamma_A\subset \Gamma$ be a finite set of curves such that the subsurface of $S$ filled by $\Gamma_A$ contains all the curves in $A.$ We then have that $h\in g\cdot U(A) = h\cdot U(A)\subseteq h\cdot U(\Gamma_A),$ completing the proof.
\end{proof}

\medskip\noindent\textbf{The mapping class group is not discrete.} We now use our characterization of the topology on $\MCG(S)$ to verify that the mapping class group is not discrete. 

Let $\Gamma$ be a stable Alexander system. 
Take $\{c_n\}$ to be a set of distinct curves in $\Gamma.$ It follows from the local finiteness of Alexander systems and our characterization of the topology of the mapping class group in Proposition \ref{prop:topology} that $\{T_{c_n}\}_{n\in\mathbb{N}}$ limits to the identity. This is so because every finite-type subsurface of $S$ is eventually fixed by the $T_{c_n}.$

This is akin to the example provided by Aramayona--Vlamis in their proof of the fact that mapping class groups of infinite-type surfaces are not discrete. Further discussion of this result is provided in Aramayona--Vlamis \cite{AV}.

\medskip\noindent\textbf{The mapping class group is not locally compact.} Aramayona--Vlamis further discuss why the mapping class group of an infinite-type surface is not locally compact; this discussion is rooted in the permutation topology. We show the same result, but using the Alexander topology.

To begin, fix a curve $c$ and consider the mapping class $T_c.$ Then, any neighborhood $N$ of $T_c$ must contain all the powers of $T_c$ (if $T_c$ fixes a curve, then so does $T^k_c$ for all $k$). As such, $\{T^k_c\}$ is contained in $N$.

Suppose $h$ is the limit of $\{T^k_c\}$. Applying Proposition \ref{prop:topology}, we see that $\{T^k_c\}$ must agree with $h$ on increasing and exhaustive subsets of $\Gamma$ (and therefore $S$). However, $\{T^k_c\}$ does not stabilize on a neighborhood of $c$: let $\gamma$ be a curve that intersects $c.$ Then, $T^k(\gamma)\neq T^j(\gamma)$, implying that $\{T^k_c\}$ is not eventually in $h\cdot U(\{\gamma\}),$ a contradiction. Therefore, $\MCG(S)$ is not locally compact. 

The same sequence is used by Aramayona--Vlamis to show that the mapping class group is not locally compact \cite{AV}. The novelty in our approach is that the Alexander method provides a simpler way to check convergence.

\section{Proof of the Alexander method}\label{sec:proof}

In this section, we prove the main theorem. Let $S$ be a surface and $\Gamma$ an Alexander system on $S$. We restate the Alexander method for convenience.

Recall that $G(S,\Gamma)$ is the surface graph on $S$ whose vertices correspond to intersections between curves and arcs in $\Gamma$ as well as endpoints of arcs and whose edges correspond to the connected components of $\cup \gamma_i \setminus V(G)$. The barycentric subdivision of $G(S,\Gamma)$ is denoted $G'(S,\Gamma).$

With that notation, the main theorem states the following. Suppose $\phi:S\to S$ is a homeomorphism that permutes the isotopy classes of curves in $\Gamma$ according to permutation $\sigma.$  Then
\begin{enumerate}
	\item  there exists a homeomorphism $\psi:S\to S$ isotopic to $\phi$ rel $\partial S$ such that $\psi(\gamma_i)=\gamma_{\sigma(i)}$ for all $\gamma_i\in \Gamma,$ 
	\item $\phi$ induces a unique automorphism $\phi_*$ of $G'(S,\Gamma),$ and 
	\item if $\Gamma$ is filling, then $\phi_*$ is the identity if and only if $\phi$ is isotopic to the identity.
\end{enumerate}

\begin{proof}[Proof of main theorem.] We prove the three statements of the Alexander method in turn.

\medskip\noindent\textit{Statement 1: finding a homeomorphism $\psi$ isotopic to $\phi$.} We first notice that $\phi(\Gamma)$ is an Alexander system. Since we will be isotoping $\phi(\Gamma)$ to $\Gamma,$ it is sufficient to consider isotoping an Alexander system to another Alexander system such that each curve (or arc) in the first is isotopic to a curve (or arc) in the second.

With that in mind, let $\Gamma$ and $\Gamma'$ be Alexander systems in $S$ such that $\gamma_i\in\Gamma$ is isotopic to $\gamma_i'\in\Gamma'.$ We will construct an isotopy of $S$ such that each $\gamma_i$ is taken to $\gamma_i'.$ 

 Both Alexander systems are indexed by $\mathbb{N}$ or $\{1,\ldots,n\}$. We present the proof in the former case and the latter case follows the same outline. 

We claim that there exists a subsurface exhaustion $\{S_i\}_\mathbb{N}$ of $S$ as follows: $S_i$ contains $\gamma_i\cup\gamma'_i.$ That is, $\gamma_i$ and $\gamma'_i$ are isotopic in $S_i$ (rel $\partial S_i$). This is so because each curve and arc exists on a finite-type subsurface. The work of Richards \cite{Richards} gives us that there exists an exhaustion of $S$ by finite-type subsurfaces, so each $\gamma_i\cup\gamma_i'$ lives on a subsurface in the exhaustion. We take a subsequence of the exhaustion to obtain $\{S_i\}$ as required above.

We will now define a sequence of isotopies that we will perform sequentially to produce an isotopy $H:S\times I\to S$ such that $H(S,0)=\id$ and $\psi(S):=H(S,1),$ with $\psi(\gamma_i)=\gamma'_i.$ 

We claim that there exists a sequence of isotopies $H_i:S\times I \to S$ supported on $S_i$ such that
\begin{enumerate}
    \item the $i$th curve or arc is corrected: $H_i(\gamma_i,1)=\gamma'_i$,
    \item previously corrected curves or arcs are not moved: for all $j<i$ and $p\in \gamma_j,$ $H_i(p,t)=p$ for all $t\in I,$ 
    \item the intersections with other curves (and arcs) are preserved: $H_i(\gamma_i\cap \gamma_j,1)=\gamma_i'\cap \gamma'_j$ for all $j$ in such a way that each connected component of $H_i(\gamma_j\setminus \gamma_i,1)$ is isotopic to a component of $\gamma'_j\setminus \gamma'_i$ rel $\gamma'_i,$ 
    \item non-intersected subsurfaces are fixed: if $(\gamma_i\cup \gamma'_i) \cap S_j=\emptyset,$ then $H_i$ is supported on $S_i\setminus S_j$ (for $j<i$), and
    \item isotopies are gluable: $H_i(x,1)=H_{i+1}(x,0)$ for all $x\in S.$
\end{enumerate}

To prove the claim, we address each of the five points. Point (1) is possible by directly applying the Alexander method for finite-type surfaces on $S_i$ \cite[Proposition 2.8]{primer}. Point (2) is achievable since we can perform the isotopy on $S_i$ cut along $\gamma_j$ for $j<i$ (retaining the boundary arising from $\gamma_j$). Such an isotopy is the identity on the boundary---and, as such, it extends to an isotopy of $S$ that fixes $\gamma_j.$ Point (3) is possible because we can perform an isotopy that rotates, stretches, and/or shrinks along $\gamma_i$ until the intersections line up. 

We now prove that point (4) is achievable. Let $\Delta_j$ be the set of nonperipheral curves in $\partial S_j.$ Then, $\Delta_j\cup \{\gamma_i\}$ and $\Delta_j\cup \{\gamma'_i\}$ are Alexander systems in $S_i$ that satisfy the hypotheses for the finite-type Alexander method. Following the proof of Lemma 2.9 in Farb--Margalit \cite{primer}, there is an isotopy of $S_i$ that takes $\gamma_i$ to $\gamma'_i$ and fixes $\Delta_j.$ As such, we may take $S_j$ to be fixed as well.

Property (5) is achieved by defining the isotopies sequentially. We have therefore completed the proof the claim.

We define a sequence of functions to reparametrize the intervals.
\begin{align*}
    \zeta_i:\left[1-\frac{1}{2^{i-1}},1-\frac{1}{2^{i}}\right] &\to [0,1]\\
    x &\mapsto 2^i \left[ x-\left(1-\frac{1}{2^{i-1}}\right)\right]
\end{align*} 

Finally, we define our overall isotopy $H:S\times I\to S$.
\[
    H(p,t)=
    \begin{cases}
    H_i(p, \zeta_i(t)) & t \in \left[ 1-\frac{1}{2^{i-1}}, 1-\frac{1}{2^i}\right]\\
    \psi(p) &t=1
    \end{cases}
\]

The map $H$ is continuous for all $(x,t)$ if $t<1$ due to the continuity of the $H_i$, point 
(5) above, and the pasting lemma from point-set topology. It remains to show that $H$ is continuous at $t=1.$ In other words, we want to show that, for all $p\in S,$ there exists an open neighborhood $U_p$ of $p$ and $t_0\in [0,1)$ such that $H(x,t)=x$ for all $x\in U_p$ for all $t>t_0$ for some $t_0.$ Suppose $p\in \interior(S_i)$ for some $i$. Let $J_0$ be the largest $j$ such that $(\gamma_{J_0}\cup\gamma'_{J_0})\cap S_i \neq \emptyset.$ Such a largest value exists due to the local finiteness of Alexander systems. Then, $H_J(x,t)|_{S_i\times I}=x$ for all $J>J_0$ and for all $t$ by point (4) above. Since $p\in \interior(S_i),$ we have that there is an open neighborhood $U_p$ of $p$ such that $H(x,t)=\psi(x)$ for all $x\in U_p$ for all $t>t_0$ for some $t_0.$

The homeomorphism $\psi(x):=H(x,1)$ has the property $\psi(\gamma_i)=\gamma_i'$ by points (1), (2), and (3) above, so $\gamma_i$ is isotoped to $\gamma'_i$ and is not moved afterward.

\medskip\noindent\textit{Statement 2: $\phi$ induces a unique automorphism of $G(S,\Gamma).$}
We now return to our original notational conventions. That is, the $\Gamma$ in the proof of statement 1 is now $\phi(\Gamma)$ and the $\Gamma'$ in the proof of statement 1 is now $\Gamma.$ It follows from Statement 1 that $\phi$ induces an automorphism of $G(S,\Gamma)$. It remains to show that this automorphism is unique. 

It is sufficient to show that any automorphism of $G(S,\Gamma)$ induced by an isotopy of $S$ is the identity. The universal cover of $S$ is the hyperbolic plane $\mathbb{H}^2$. The preimage of $G(S,\Gamma)$ in $\mathbb{H}^2$ is a graph whose vertices are the preimages of vertices of $G(S,\Gamma)$ and edges are preimages of edges in $G(S,\Gamma).$ An isotopy of $S$ lifts to an isotopy of $\mathbb{H}^2$. 

As each arc and curve is isotoped on a finite-type surface, the induced graph automorphism of the preimage of $G(S,\Gamma)$ must preserve the endpoints of the lifts of the arcs and curves at $\partial \mathbb{H}^2$, and therefore preserve every elevation of every curve and arc in $\Gamma.$ As such, all intersections between elevations of curves in $\Gamma$ must be preserved by the isotopy, and it follows that the induced automorphism of the preimage graph in $\mathbb{H}^2$ is the identity. We conclude that $\phi$ induces a unique automorphism of $G(S,\Gamma).$

\medskip\noindent\textit{Statement 3: inducing the identity automorphism on $G'(S,\Gamma)$ implies being isotopic to the identity.}
Let $\Gamma$ be a filling Alexander system in $S$. Suppose $\phi_*$ is the identity automorphism of $G'(S,\Gamma)$; that is, such that $\phi_*$ preserves the orientation of loop edges in $G(S,\Gamma)$. By our work above, $\phi_*$ takes vertices to vertices and edges to edges; thus, if an edge has distinct vertices, its orientation is preserved by $\phi_*.$ Then, up to isotopy, $\phi$ fixes each (potentially once-punctured or M\"obius band) disk in $S\setminus \Gamma$. Furthermore, $\phi$ fixes, with orientation, each edge (and thus the boundary) of each disk. It follows that, up to isotopy, $\phi$ is equal to the identity on every disk, once-punctured disk, and M\"obius band. We then apply the pasting lemma from point-set topology to obtain that $\phi$ is isotopic to the identity.
\end{proof}

\section{Non-example of an Alexander system}\label{sec:nonexamples}

In this section, we illustrate why the local finiteness condition on Alexander systems is necessary.

The main idea of this restriction is that we do not want any limit points of distinct elements of our Alexander system. That is, any sequence $\{x_n\}$ of points on distinct $\gamma_i$ does not converge. Let us consider an example with seemingly minimal convergence.

Let $S$ be the plane with punctures at the origin and at angles $\frac{1}{n}$ and $0$ on the unit circle. Let $\Gamma$ consist of arcs connecting 0 to the punctures on the unit circle (except to the puncture at angle 0). It suffices to show that $\Gamma$ is not well-defined up to homeomorphism. 

Let $\Gamma_1$ be the collection of arcs that appear as straight rays from the origin to the punctures on the unit circle (except at angle 0). Let $\Gamma_2$ be the same collection of arcs, except the arc at angle $1$ is isotoped counterclockwise to hit the positive $x$-axis before $x=1$. As such, there is a sequence of points on the arcs in $\Gamma_2$ that converges to a point in an arc; this is a property preserved by homeomorphism and it is not true of $\Gamma_1.$ An example of such arrangements is shown in Figure~\ref{fig:localfinitenessnonexample}.

\begin{figure}[h]
    \centering
	\includegraphics[width=\textwidth]{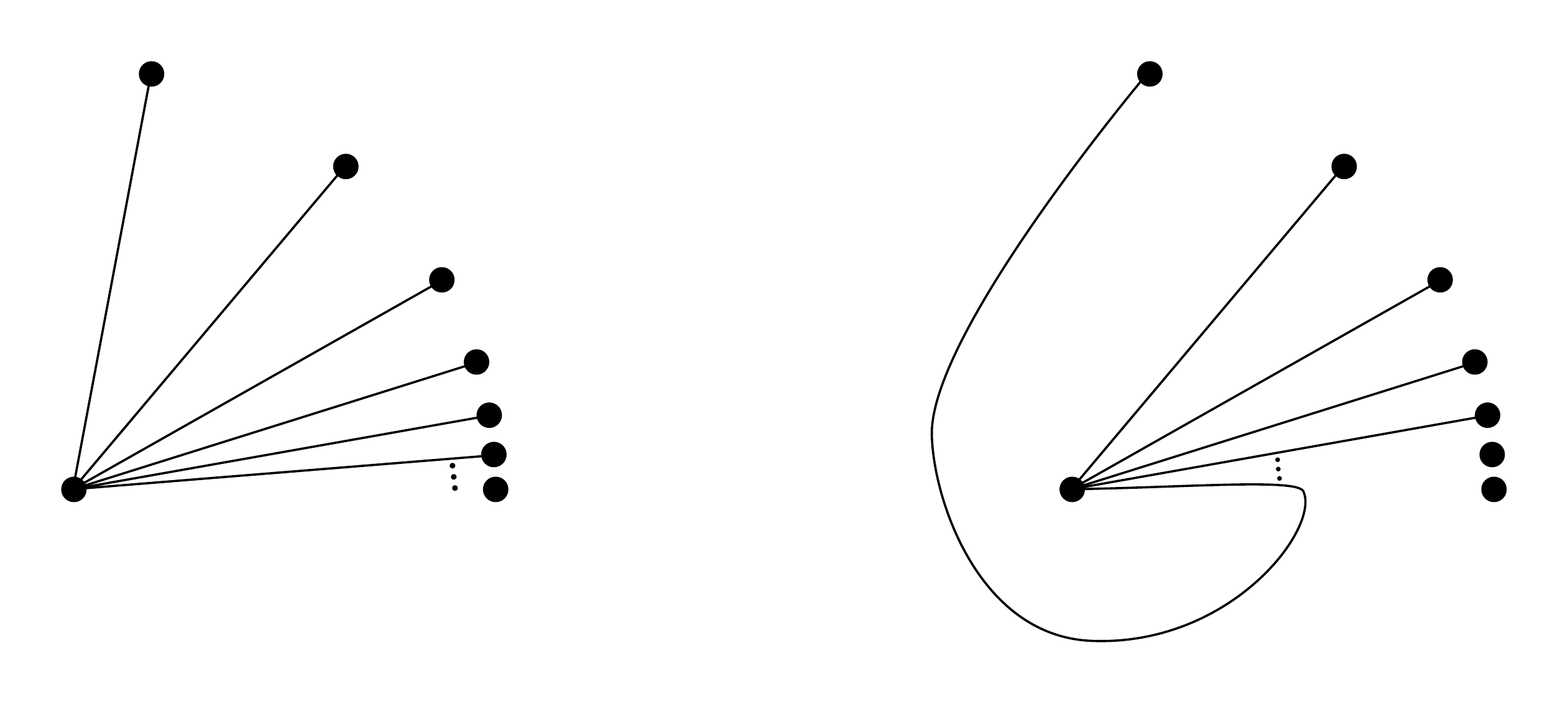}
    \caption{An example of an Alexander system $\Gamma$ on a surface $S$ where local finiteness is broken. We note that any disk about the origin (isolated puncture connected to all other punctures) intersects infinitely many arcs in $\Gamma.$ The key point is that although all arcs in the left image are isotopic to their corresponding arcs in the right image, the two configurations are not homeomorphic.}
    \label{fig:localfinitenessnonexample}
\end{figure}

We note that we require that arcs be of finite type in order for the above proof technique to work. However, there may be a version of the Alexander method which allows for all arcs.

\bibliography{AlexMethodPaper}{}
\bibliographystyle{plain}

\end{document}